\documentclass[11pt,twoside]{amsart}

\usepackage{setspace, a4wide}
\usepackage[T1]{fontenc}
\usepackage[english]{babel}
\usepackage{amssymb,amsfonts,amsthm,amsmath,mathrsfs}
\usepackage{layout,version,enumerate}
\usepackage[all,matrix]{xy}
\usepackage[varg]{pxfonts}
\usepackage{indentfirst}
\usepackage{url}
\usepackage{color}
\usepackage[colorlinks=true,linkcolor=blue,citecolor=magenta]{hyperref}
\usepackage[numbers]{natbib}

\theoremstyle{plain}
\newtheorem{thm}{Theorem}
\newtheorem{prop}[thm]{Proposition}

\newtheorem{lem}[thm]{Lemma}
\newtheorem*{thm*}{Theorem}
\newtheorem*{fact*}{Fact}

\theoremstyle{definition}
\newtheorem{dfn}[thm]{Definition}
\newtheorem*{dfn*}{Definition}

\theoremstyle{remark}

\newcommand{\bn}{\mathbb{N}}

\newcommand{\ul}{\mathfrak u}

\newcommand{\cp}{\mathcal P}

\newcommand{\eps}{\varepsilon}
\newcommand{\id}{\mathfrak{id}}

\renewcommand{\id}{\mathrm{1}}
\newcommand{\haar}{\lambda}
\def\defin#1{\textbf{\emph{#1}}}  
\newcommand{\pmp}{p.m.p.\ }

\usepackage{tikz}
\usetikzlibrary{matrix,arrows}

\renewcommand{\part}{\operatorname{Part}}

\newcommand{\lcsc}{l.c.s.c.\ }

\newcommand{\covol}{\operatorname{covol}}

\newcommand{\Bibkeyhack}[1]{}

\begin{document}

\onehalfspace

\thispagestyle{empty}
\title{On Farber sequences in locally compact groups}

\author{Alessandro Carderi}
\address{A.C., Institut f\"ur Geometrie, TU Dresden, 01062 Dresden, Germany}
\email{alessandro.carderi@tu-dresden.de}

\begin{abstract}
{We will prove that any sequence of lattices in a fixed locally compact group which satisfy the conclusion of the Stuck-Zimmer theorem is Farber.}
\end{abstract}

\maketitle
 
\section*{Introduction}

A sequence of lattices $\{\Gamma_n\}_n$ of a locally compact second countable (\textit{l.c.s.c.}) group $G$ is \defin{Farber} (or \textit{almost everywhere thick}) if for every precompact neighborhood of the identity $U\subset G$ we have that \[\lim_n \frac{\haar\left(\left\{g\Gamma_n\in G/\Gamma_n:\ g\Gamma_ng^{-1}\cap U\neq \{\id_G\}\right\}\right)}{\haar(G/\Gamma_n)}=0,\]
where $\haar$ denotes a Haar measure on $G$. Farber sequences of finite index subgroups were introduced by Farber in \cite{far1998} in order to understand the condition that a chain of finite index subgroups has to satisfy to make the L\"uck approximation for $\ell^2$-Betti numbers hold. One can think of a Farber sequence as a sequence of subgroups for which the quotients $G/\Gamma_n$ (and the action of $G$ on them) approximate the group. Indeed in the case of countable groups a sequence is Farber if and only if the sequence of quotients is a sofic approximation of the group. 

Recently it was proven in \cite{abe2017}, \cite{gel2017} and \cite{lev2017a} that in some situations some sequences of lattices are automatically Farber. All the known cases of such a phenomenon are using a form of the Stuck-Zimmer theorem. In \cite[Section 4.3]{CGS} we pointed out that using the notion of ultraproduct one can give a very easy and self-contained argument to show that any sequence of finite index subgroups of a \textit{Stuck-Zimmer} group is Farber. The aim of this note is to use the notion of regular ultraproduct, that we recently introduced in \cite{ulcg}, to derive and generalize the above mentioned theorems of \cite{abe2017}, \cite{gel2017} and \cite{lev2017a}.

\begin{dfn}
  A group $G$ has \defin{property }$\{N_j\}_{j\leq k}$-\defin{(SZ)} \textit{with respect some normal subgroups} $\{N_j\}_{j\leq k}$ of $G$ if the stabilizers of every probability measure preserving action of $G$ on a standard Borel space\footnote{cf. Proposition \ref{step4}} are finite, whenever
  \begin{itemize}
  \item every $G$-orbit is a null set,
  \item the action of every subgroup $N_i$ has spectral gap.
  \end{itemize}
\end{dfn}

We will just say that a group $G$ has property (SZ) if we do not want to specify the normal subgroups. The definition is inspired by a theorem of Stuck and Zimmer \cite{stu1994} which says that every semisimple Lie group whose simple factors have rank at least 2 has property (SZ). In their context the group $G$ has property (T) and hence the condition on spectral gap is automatic. The current form is also inspired by Theorem 4 of \cite{lev2017a} which, as it is stated in the paper, is a formal corollary of \cite{bad2006}. The theorem claims that products of simple locally compact groups (which may not have property (T)) has property (SZ). In \cite{bad2006} it is also proven that products of compactly generated property (T) have property (SZ). Finally the work of Stuck and Zimmer was extended to semisimple analytic groups whose simple factors have rank at least 2 in \cite{lev2017}. In all the above cases we have that $G=N_1\times \ldots \times N_k$. We do not know whether it is always the case.

Recall that a measurable, measure preserving action of a \lcsc group $G$ on a probability space $(X,\mu)$ has spectral gap if the Koopman representation $\mathrm L_0^2(X,\mu)$ on the space of functions with integral $0$ has no almost invariant vector. A sequence of actions of $G$ on $\{(X_n,\mu_n)\}_n$ has spectral gap if the representation $\bigoplus_n \mathrm L_0^2(X_n,\mu_n)$ has no almost invariant vector. 

\begin{dfn}
  Let $G$ be a \lcsc group and let $\{N_j\}_{j\leq k}$ be normal subgroups of $G$. We say that $G$ has the \defin{property} $\{N_j\}_{j\leq k}$-($\boldsymbol{\tau}$) \textit{with respect to the sequence of lattices} $\{\Gamma_n\}_n$ if for each $j\leq k$ we have that $\{(G/\Gamma_n,\mu_n)\}_n$ has spectral gap as a sequence of $N_j$-spaces, where we denote by $\mu_n$ the renormalized Haar measure on $G/\Gamma_n$.
\end{dfn}

Observe that if an action has spectral gap, then it is ergodic. So in particular if $G$ has the property $\{N_j\}_{j\leq k}$-($\tau$) with respect to the sequence of lattices $\{\Gamma_n\}_n$, then each $\Gamma_n$ is \textit{irreducible} (with respect to $\{N_j\}_j$). The converse hold whenever each $N_j$ has property (T). The aim of this note is to prove a generalization of Theorem 4.4 of \cite{abe2017}, of the main theorem of \cite{gel2017} and \cite{lev2017a}.

\begin{thm}
  Let $G$ be a \lcsc group and let $N_j<G$ be normal subgroups for $j\leq k$. Assume that $G$ has property $\{N_j\}_{j\leq k}$-(SZ) and that $G$ has the property $\{N_j\}_{j\leq k}$-($\tau$) with respect to the sequence of lattices $\{\Gamma_n\}_n$ for which the covolume tends to infinity, $\lim_\ul\covol(\Gamma_n)=\infty$. Then the action of $G$ on the regular ultraproduct $[G/\Gamma_n]_\ul^R$ is ergodic and all the stabilizers are finite.

In particular, $\{\Gamma_n\}_n$ is \textit{mostly Farber}\footnote{see Definition \ref{dfn:mfarb}} and if $G$ has a neighborhood of the identity without discrete subgroups, then $\{\Gamma_n\}_n$ is Farber.  
\end{thm}

We would like to recall that the \textit{regular ultraproduct} $[X_n]_\ul^R$ of a sequence of probability measure preserving actions on $(X_n,\mu_n)$ of the \lcsc group $G$ is the set of sequences $x_n\in X_n$ modulo the equivalence relation $\sim_\ul^R$ defined by $(x_n)_n\sim_\ul^R(y_n)_n$ if there exists a sequence $h_n\in G$ such that $\lim_\ul h_n=\id$ and $h_nx_n=y_n$ for $\ul$-almost every $n$. The $\sigma$-algebra of the regular ultraproduct is generated by sequences of subsets $\{A_n\subset X_n\}_n$ which are \textit{regular}, that is such that for every $\eps>0$ there exists a neighborhood of the identity $U\subset G$ such that $\mu_n(UA_n\setminus A_n)\leq \eps$ for $\ul$-almost every $n$. The measure of an equivalence class of a regular sequence $[A_n]_\ul^R$ is $\mu_\ul^R([A_n]_\ul^R)=\lim_\ul \mu_n(A_n)$. The proof of the theorem will use some of the easy properties we collected about the regular ultraproduct in \cite{ulcg}. We will sketch the proofs of these facts wherever they are used.

In order to prove the theorem we will follow the following steps. 

\begin{enumerate}
\item The action of $G$ on $[G/\Gamma_n]_\ul^R$ has finite stabilizers if and only if the sequence of subgroups $\{\Gamma_n\}_n$ is mostly Farber.
\item Property $\{N_j\}_{j\leq k}$-($\tau$) implies that the action of $N_j$ on $[G/\Gamma_n]_\ul^R$ has spectral gap for every $j\leq k$.
\item The action of $G$ on $[G/\Gamma_n]_\ul^R$ has null orbits.
\item There exists a standard factor of $[G/\Gamma_n]_\ul^R$ which allows us to conclude the proof.
\end{enumerate}

\section*{Proof}

\subsection*{Step 1: Farber vs Freeness}

Let us fix a \pmp action of a group $G$ on the probability space $(X,\mu)$. We say that the action is \defin{free} if it is \textit{point-wise almost everywhere free}, that is for every $g\in G$ and for almost every point $x$ we have that $gx\neq x$. Observe that if the group is locally compact and the action is measurable, Fubini implies that we can exchange ``for every $g$'' and ``for almost every $x$''. 

We say that an action is \defin{measurably free} if for every $g\in G$ there exists a partition $\{A_n\}_n$ of a conull subset of $X$ such that $gA_n\cap A_n=\emptyset$. Clearly measurably free implies free and for actions on standard Borel spaces the two notions coincide. 

\begin{dfn}\label{dfn:mfarb}
  Let $G$ be a \lcsc group and let $\{\Gamma_n\}_n$ be a sequence of lattices of $G$. We say that $\{\Gamma_n\}_n$ is \defin{mostly Farber} if for every precompact neighborhoods of the identity $U\subset V$ we have  \[\lim_n\frac{\lambda\left(\left\{ g\Gamma_n\in G/\Gamma_n:\ g\Gamma_ng^{-1}\cap V\subset U \right\}\right)}{\lambda(G/\Gamma_n)}=1.\] 
\end{dfn}

In \cite{ulcg} we observed that the stabilizer of a point $[x_n]_\ul^R\in [G/\Gamma_n]_\ul^R$ is the pointed Hausdorff limit of the stabilizers of $x_n$ as subsets of $G$. Indeed we have that $g[x_n]_\ul^R=[x_n]_\ul^R$ if and only if there exists a sequence $\{h_n\}$ such that $\lim_\ul h_n=\id_G$ and $gh_nx_n=x_n$ for $\ul$-almost every $n$. From this observation, we can easily derive the following proposition. 

\begin{prop}\label{step1}
 Let $G$ be a \lcsc group and let $\{\Gamma_n\}_n$ be a sequence of lattices of $G$ whose covolume tends to infinity. Then the action of $G$ is free if and only if the sequence of lattices is mostly Farber. 
\end{prop}

In some cases it is easy to see that a mostly Farber sequence is automatically Farber. 

\begin{lem}\label{mostly}
  Let $G$ be a \lcsc group and let $\{\Gamma_n\}_n$ be a mostly Farber sequence of lattices of $G$. Assume that one of the following conditions holds
  \begin{itemize}
  \item the sequence $\{\Gamma_n\}_n$ is nowhere thin\footnote{a sequence of lattices $\{\Gamma_n\}_n$ is nowhere thin if there exists a neighborhood of the identity $U\subset G$ such that for every $g\in G$ and $n\in\bn$ we have $g\Gamma_n g^{-1}\cap U=\{\id\}$};
  \item there is a neighborhood of the identity $U\subset G$ which does not contain any discrete subgroup;
  \end{itemize}
then $\{\Gamma_n\}_n$ is a Farber sequence.
\end{lem}
\begin{proof}
  The first point is obvious from the definition. For the second consider $V:=U^2$. Take $g\Gamma_n\in G/\Gamma_n$ such that $S:=g\Gamma_ng^{-1}\cap V\subset U$. Then the group generated by $S$ is a subgroup of $g\Gamma_ng^{-1}$. On the other hand observe that $S^2\subset g\Gamma_ng^{-1}\cap V\subset g\Gamma_ng^{-1}\cap U=S$, thus $S$ is a discrete subgroup of $U$ and hence $S=\{\id\}$.  
\end{proof}

\subsection*{Step 2: Spectral gap vs Property ($\tau$)}

\begin{prop}\label{step2}
  Let $G$ be a \lcsc group and assume that it acts measurably and preserving the measure on the sequence of probability spaces $\{(X_n,\mu_n)\}_n$. Assume moreover that $\{(X_n,\mu_n)\}_n$ has spectral gap. Then the action of $G$ on the regular ultraproduct $[X_n]_\ul^R$ has spectral gap. 
\end{prop}
\begin{proof}
  Assume that $\{f_\ul^k\}_k$ is a sequence of almost invariant vectors of $\mathrm L_0^2([X_n]_\ul^R,\mu_\ul^R)$. Since we have that $\mathrm L_0^2([X_n]^R_\ul,\mu^R_\ul)$ is a subspace of the metric ultraproduct of the Hilbert spaces $\mathrm L_0^2(X_n,\mu_n)$, we have that each $f^k_\ul$ is represented as a sequence of functions $f_n^k\in \mathrm L_0^2(X_n,\mu_n)$, that is $f_\ul^k([x_n]_\ul^R)=\lim_\ul f_n^k(x_n)$ for almost all $[x_n]_\ul^R\in [X_n]_\ul^R$. The fact that $f^k_\ul$ is defined on the regular ultraproduct tells us that for every $\eps>0$ there exists a neighborhood of the identity $U_\eps\subset G$ such that for $\ul$-almost every $n$ we have $\|f_n^k-gf_n^k\|_2\leq \eps$ for every $g\in U_\eps$.

Since $\{(X_n,\mu_n)\}_n$ has spectral gap, there is a compact subset $K\subset G$ and an $\eps>0$ such that $\bigoplus \mathrm L_0^2(X_n,\mu_n)$ has no $(K,\eps)$-invariant vectors. Choose $U:=U_{\eps/2}$ as above. Then there is a finite set $\{g_j\}_{j\leq l}$ such that $\cup_j Ug_j \supset K$. Since $\{f_\ul^k\}_k$ is a sequence of almost invariant vectors we have that there exists $k$ such that we have $\|f_\ul^k-gf_\ul^k\|_2\leq \eps/2$ for every $g\in K$. This implies that for $\ul$-almost every $n$ and every $j\leq l$ we have $\|f_n^k-g_jf_n^k\|_2\leq \eps/2$. Finally observe that if $g\in K$, then there exists $u\in U$ and $j\leq l$ such that $g=ug_j$ and we have that $\|f_n^k-gf_n^k\|_2\leq \|f_n^k-uf_n^k\|_2+\|f_n^k-g_jf_n^k\|\leq \eps$ for $\ul$-almost every $n$ which is a contradiction.
\end{proof}

\subsection*{Step 3: Null orbits}

Let $G$ be a \lcsc group and assume that it acts measurably and preserving the measure on the probability space $(X,\mu)$. We say that the action has \defin{null orbits} if whenever we fix a neighborhood of the identity $V\subset G$, for every $\eps$ there exists a measurable subset $A_\eps\subset X$ such that $GA_\eps\subset X$ is conull, $VA_\eps$ is measurable and $\mu(VA_\eps)\leq \eps$. Note that $A_\eps$ could have measure $0$. Let us give some examples.

\begin{itemize}
\item If the action of $G$ on $(X,\mu)$ is free and it admits an external cross section, then the action has null orbits and the subsets $A_\eps$ can be chosen as images of subsets of the external cross section.
\item If $(X,\mu)$ is a standard Borel space, then the action has null orbits if and only if each orbit is a null set. This follows, for example, from the existence of a discrete section.
\item If an action is measurably free, then it has null orbits. Indeed one can show that every measurably free action has a free standard factor, see \cite[Proposition 2.13]{ulcg} and \cite[Theorem 3.28]{CGS}.
\end{itemize}

\begin{prop}\label{step3}
 Let $G$ be a compactly generated \lcsc group and let $\{\Gamma_n\}_n$ be a sequence of lattices of $G$ whose covolume tends to infinity. Then the action of $G$ on the regular ultraproduct $[G/\Gamma_n]_\ul^R$ has null orbits.
\end{prop}
\begin{proof}
  Fix a right-invariant compatible metric $d$ on $G$. Denote by $B_\eps$ the $d$-ball of radius $\eps$ around the identity and assume that $U:=B_1$ is precompact. Take a maximal (under inclusion) $U$-separated subset $D_n\subset G/\Gamma_n$, that is a subset $D_n$ such that for every $d\neq d'\in D_n$ we have $Ud\cap Ud'=\emptyset$ and $U^4D_n=G/\Gamma_n$. We claim that for every $i\geq 1$ there exists $D^i_n\subset D_n$ such that $\mu_n(UD^i_n)\leq 2^{-i}\mu_n(D_n)$ for every $n$ and for which there exists a compact subset $K_i\subset G$ such that $K_iD_n^i=G/\Gamma_n$. Let us define $D^1_n$ and the other can be easily obtained by induction. Let $K\subset G$ be a symmetric compact subset. We consider the graph $(D_n,E_n^K)$ defined by $(d,d')\in E_n$ if $Kd\cap Kd'\neq \emptyset$. We observe that for $K$ big enough the graph has no isolated points. We consider a partition of $D_n$ in subsets consisting on 2 or 3 elements of $D_n$ which are contained in a translate of $K^3$. The set $D^1_n$ will be chosen to contain a point from each atom $\alpha$ of the partition, the point $x\in \alpha$ such that $Ux$ has minimal measure among all $y\in\alpha$.

Observe now that we do not know whether $\{D_n^i\}_n$ is a regular sequence. However we have that $ED_n^i$ is measurable for every $E\subset G$ Borel. In particular $B_\eps D_n^i$ is measurable and observe that the function $\eps\mapsto \lim_\ul(B_\eps D_n^i)$ is monotone and hence continuous almost everywhere. This implies, see Lemma 1.14 of \cite{ulcg}, that there are $\eta,\eps\leq 1/2$ such $\{B_\eta D_n^i\}_n$ and $\{B_\eps B_\eta D_n^i\}_n$ are regular and such that $B_\eps[B_\eta D_n^i]_\ul^R=[B_\eps B_\eta D_n^i]_\ul^R$. Observe also that $\mu_n(B_\eps B_\eta D_n^i)\leq \mu_n(B_1 D_n^i)\leq 2^{-i}$ and therefore the action of $G$ on the regular ultraproduct has null orbits.
\end{proof}

\subsection*{Step 4: Non-standard Stuck-Zimmer}

\begin{prop}\label{step4}
  If $G$ has property $\{N_j\}_{j\leq k}$-(SZ), then almost every stabilizer of every measurable action of $G$ on a probability space $(X,\mu)$ is finite whenever
  \begin{itemize}
  \item the action has null orbits,
  \item the action of every subgroup $N_i$ has spectral gap.
  \end{itemize}
\end{prop}
\begin{proof}
  This follows from Mackey's theorem \cite{mac1962}. Indeed let $\{A_n\}_n$ be measurable subsets of $X$ which witness that the action has null orbits, that is for some fixed neighborhood of the identity $V$ we have that $\mu(VA_n)$ tends to $0$. Consider the $G$-invariant $\sigma$-algebra generated by them. Since $G$ is separable this algebra is separable, see for example \cite[Lemma 1.4]{ulcg}. Then Mackey's theorem combined with \cite[Proposition B.5]{Zimmer}, see \cite[Proposition 2.12]{ulcg}, implies that the action of $G$ on $X$ has a standard factor $Y$ which has null orbits. Clearly the action of $N_i$ on $Y$ has still spectral gap and hence almost every stabilizer of the action on $Y$ is finite, which implies the desired result.
\end{proof}

\subsection*{Conclusions}

Let $G$ be a \lcsc group with property $\{N_j\}_{j\leq k}$-(SZ). Assume that $G$ has the property $\{N_j\}_{j\leq k}$-($\tau$) with respect to the sequence of lattices $\{\Gamma_n\}_n$ and assume that  $\lim_\ul\covol(\Gamma_n)=\infty$. Consider the action of $G$ on the regular ultraproduct $[G/\Gamma_n]_\ul^R$. By assumption for every $j$ we have that $\{(G/\Gamma_n,\mu_n)\}_n$ has spectral gap as a sequence of $N_j$-spaces, therefore Proposition \ref{step2} implies that the action of $N_j$ on $[G/\Gamma_n]_\ul^R$ has spectral gap. Proposition \ref{step3} tells us that the action of $G$ on $[G/\Gamma_n]_\ul^R$ has null orbits and thanks to Proposition \ref{step4} the action of $G$ on $[G/\Gamma_n]_\ul^R$ has finite stabilizers. Therefore Proposition \ref{step1} tells us that the sequence $\{\Gamma_n\}_n$ is mostly Farber and Lemma \ref{mostly} gives us the conditions for which the sequence is actually Farber.

\section*{Acknowledgements}

This research was supported by the ERC Consolidator Grant No. 681207.




\bibliographystyle{alpha}
\bibliography{bib}

\end{document}